\documentclass[12pt]{article}
\usepackage{amssymb}
\usepackage[T1]{fontenc}
\usepackage[utf8]{inputenc}
\usepackage[intlimits]{amsmath}
\usepackage{enumerate}
\usepackage[dvips]{graphicx}
\usepackage[all]{xy}

\newtheorem{theorem}{Theorem}
\newtheorem{corollary}{Corollary}
\newtheorem{proof}{Proof}

\newtheorem{example}{Example}

\usepackage{multirow}

\usepackage[colorlinks]{hyperref}

\begin{document}

\begin{center}\Huge
Darboux transformations and second order difference equations 
\end{center}
\begin{center}\large
Alina Dobrogowska${}^*$  and David J. Fern\'andez C.${}^\dagger$, 
\end{center}\begin{center}
${}^*$ Institute of Mathematics, University of Białystok,\\
 Ciołkowskiego 1M, 15-245 Białystok, Poland,\\[6pt]
${}^\dagger$ Physics Department, Cinvestav \\ 
AP 14-740, 07000 Mexico City, Mexico
\end{center}
\begin{center}
E-mail: alina.dobrogowska@uwb.edu.pl and david@fis.cinvestav.mx
\end{center}
\noindent {\bf Abstract.}
In this paper we implement the Darboux transformation, as well as an analogue of Crum's theorem, for a 
discrete version of Schr\"odinger equation. The technique is based on the use of first order operators 
intertwining two difference operators of second order. This method, which has been applied successfully 
for differential cases, leads also to interesting non trivial results in the discrete case. The 
technique allows us to construct the solutions for a wide class of difference Schr\"odinger equations. 
The exact solutions for some special potentials are also found explicitly.

\vskip0.5cm

\noindent {\bf Keywords:}
difference equations, discrete Schr\"odinger equation, discrete Darboux transformation, 
discrete Crum's theorem, factorization method, discrete intertwining

\section{Introduction}

The aim of this paper is to apply the factorization method to second order difference equations. We 
want to introduce an analogue of the Darboux transformation \cite{Dar} and Crum's theorem for discrete 
equations of Schr\"odinger type. Our method offers the possibility of finding new solutions for a class 
of discrete versions of Schr\"odinger equation. This article is an extension of previous works 
concerning the factorization method applied to second order differential and difference operators 
\cite{Schr, InHu, Mielnik-2,  18,    b1, bf13, B6,  B4, B5, GolOdz, B2}. Some results obtained in 
\cite{DoHo, DoJa,  MieNie} are used and adapted to our context. As general references related to these 
subjects, we recommend \cite{bf14,Dong-1,  fh99,ff05,fe10,Mielnik-3, 16}.

Specifically, we are going to show that for discrete systems the new potential is determined  either by 
one {\it shifted eigenfunction} of the initial Schr\"odinger operator for Darboux transformation 
or by $k$ such eigenfunctions in Crum's case. Our results differ from those recently obtained in 
\cite{ZPZ18}, since along this paper we assume that the initial and final operators necessarily 
have Schr\"odinger form, which is not the case in \cite{ZPZ18}. Note also that, if instead of 
shifted eigenfunctions we would use non shifted ones to implement the transformation, we would arrive to 
trivial results (see discussion at Section 3, after Theorem 1).

This paper is organized as follows. In section 2 we will present information about finite difference 
calculus and finite difference equations which is relevant for our problem. Also, we will make a brief 
survey on Darboux transformation and its iterations in the continuous case, the last ones leading to 
Crum's formulas for the potential and eigenfunctions of the final Schr\"odinger operator. In Section 3 
we discuss the general formalism for discrete Darboux transformation and we address some simple 
examples. Finally, in section 4 we explore, in general, the discrete Crum's transformation, and we 
illustrate the technique through a physically interesting example.  

\section{Preliminaries}

In this section, we recall some basic facts about difference calculus, difference equations, 
factorization method and higher-order supersymmetric quantum mechanics.

The techniques for solving differential equations based on numerical approximations were developed 
before programmable computers existed. One of the best known approaches is the Euler method, the 
simplest numerical algorithm for solving a first order differential equation. This method can be 
extended to other procedures, e.g. the Runge--Kutta methods. Starting from a differential equation, 
we replace the derivative $\psi'$ by its finite difference approximation
\begin{equation}
 \hspace*{-2cm} \psi'(t)\approx \frac{\psi(t+\Delta t)-\psi(t)}{\Delta t}=\frac{\psi(t_{n+1})-
\psi(t_{n})}{t_{n+1}-t_{n}}. 
 \end{equation}
If the step is equal to one the Taylor's theorem give us the following fundamental relation (see 
\cite{Boole})
\begin{equation}
\psi(t+1)-\psi(t)=\psi^{'}(t)+\dfrac{\psi^{''}(t)}{2!}+\dfrac{\psi^{'''}(t)}{3!}+\dots
\end{equation}

The above equation yields the standard definition of the forward difference and shift operators
 \begin{align}
 \label{1} &\bigtriangleup  \psi(n)=\psi(n+1)-\psi(n),\\
 & T^{\pm} \psi(n)= \psi(n\pm 1).\nonumber
 \end{align}
The product rule for the forward difference operator reads 
$\bigtriangleup \big(\psi\varphi \big)(n)=\psi(n)\bigtriangleup \varphi (n) +\varphi (n+1)
\bigtriangleup \psi(n)$.
The operators $T^{+}$, $\bigtriangleup$ and $\frac{d}{dx}$ are connected by the relations
\begin{equation}
T^{+}={\bf{1}}+\bigtriangleup=e^{\frac{d}{dx}}.
\end{equation}
So $\bigtriangleup$ is the fundamental operation in the calculus of finite differences. The second 
difference is given by
\begin{equation}
\bigtriangleup^2  \psi(n)=\psi(n+2)-2\psi(n+1)+\psi(n).
\end{equation}
A homogeneous linear second order difference equation can be written as follows 
\begin{equation}
\psi(n+2)+a(n)\psi(n+1)+b(n)\psi(n)=0,
\end{equation}
where $\{a\}$ and $\{b\}$ are sequences.
It is well known how to solve the above equation when the coefficients are constant (see e.g. 
\cite{El}). Let us denote by $\ell(\mathbb{Z};\mathbb{C})$ the set of complex-valued sequences. We want 
to apply the factorization method to the second order difference operator of Schr\"odinger type 
$H : \ell(\mathbb{Z};\mathbb{C}) \longrightarrow \ell(\mathbb{Z};\mathbb{C})$ given by
\begin{equation}
\label{ham}
H = -\bigtriangleup^2+V(n),
\end{equation}
where $\{V\}$ is a real-valued sequence. 

There are different approaches for discretizing the one-dimensional time-independent Schr\"odinger 
equation
\begin{equation}
\left(-\dfrac{d^2}{dx^2}+V(x)\right)\psi(x)=\lambda\psi(x).
\end{equation}
Very often a discretization appears in the standard central difference formula with the step $h$ for the 
second derivative
\begin{equation}
\left( -\dfrac{1}{h^2}\bigtriangleup \bigtriangledown+V(n)\right)\psi(n)=
\end{equation}
$$
-\dfrac{1}{h^2}\left(\psi(n+1)-2\psi(n)+\psi(n-1)\right)+V(n)\psi(n)=\lambda \psi(n),
$$
with the backward difference operator being defined by (see \cite{BoKl, BrKa, DoJa, DoHo}) 
\begin{equation}
\bigtriangledown \psi(n)=\psi(n)-\psi(n-1).
\end{equation}
The exact discretization of the Schr\"odinger equation was proposed in \cite{Tar} based on Fourier 
transforms. The study of some operators of type (\ref{ham}) starting from spectral data was done at 
\cite{Mau-Mir-1, Mau-Mir-2}. In addition, other special cases, as the factorization of Jacobi operators, 
were also investigated \cite{GeTe,  Te}.

On the other hand, in the continuous case supersymmetric quantum mechanics aims to construct a new 
Hamiltonian departing from an initial solvable one through what is called intertwining operator 
technique \cite{ff05,fe10}. In the simplest case involving first order intertwining operators, the key 
is to fulfill the following relations
\begin{equation}
H_1 A_1^+ = A_1^+ H_0 \quad \Leftrightarrow \quad H_0 A_1^- = A_1^- H_1,
\end{equation}
where
\begin{eqnarray}
&& H_i = - \frac{d^2}{dx^2} + V_i(x), \quad i=0,1, \\
&& A_1^\pm = \mp \frac{d}{dx} + f_1(x,\epsilon_1) .
\end{eqnarray}
It turns out that the superpotential $f_1(x,\epsilon_1)$ must satisfy the following Riccati equation 
associated to the factorization energy $\epsilon_1$:
\begin{eqnarray}\label{iRe}
&& f_1'(x,\epsilon_1) + f_1^2(x,\epsilon_1) = V_0(x) - \epsilon_1 .
\end{eqnarray}
If $f_1(x,\epsilon_1) = \psi_{0\,1}'(x)/\psi_{0\,1}(x)$, this equation is transformed into its 
equivalent Schr\"odinger equation:
\begin{eqnarray}\label{iSe}
- \psi_{0\,1}''(x) + V_0(x) \psi_{0\,1}(x) = \epsilon_1 \psi_{0\,1}(x) .
\end{eqnarray}
In $\psi_{0\,1}$ the first index labels the potential for the corresponding Hamiltonian while the second 
refers to the associated factorization energy. Moreover, whenever $f_1(x,\epsilon_1)$ or 
$\psi_{0\,1}(x)$ have been found, the final potential is determined by:
\begin{eqnarray}
&& V_1(x) = V_0(x) - 2 f_1'(x,\epsilon_1) = V_0(x) - 2 [\ln \psi_{0\,1}(x)]''. 
\end{eqnarray} 

This transformation can be iterated, by looking for a new Hamiltonian $H_2$ departing from $H_1$ as 
follows:
\begin{equation}
H_2 A_2^+ = A_2^+ H_1 \quad \Leftrightarrow \quad H_1 A_2^- = A_2^- H_2 ,
\end{equation}
where
\begin{eqnarray}
&& H_2 = - \frac{d^2}{dx^2} + V_2(x), \\
&& A_2^\pm = \mp \frac{d}{dx} + f_2(x,\epsilon_2).
\end{eqnarray}
Now we have to solve either the new Riccati equation, 
\begin{eqnarray}\label{nRe}
&& f_2'(x,\epsilon_2) + f_2^2(x,\epsilon_2) = V_1(x) - \epsilon_2,
\end{eqnarray}
or its equivalent Schr\"odinger equation, which appears by assuming that $f_2(x,\epsilon_2) = 
\psi_{1\,2}'(x)/\psi_{1\,2}(x)$:
\begin{eqnarray}\label{nSe}
&& - \psi_{1\,2}''(x) + V_1(x) \psi_{1\,2}(x) = \epsilon_2 \psi_{1\,2}(x).
\end{eqnarray}
Note that the solution $f_2(x,\epsilon_2)$ to equation (\ref{nRe}) can be found from two solutions 
$f_1(x,\epsilon_1), \ f_1(x,\epsilon_2)$ to the initial Riccati equation~(\ref{iRe}) for the 
factorization energies $\epsilon_1, \ \epsilon_2$ through the finite difference formula
\cite{fhm98,MieNie}:
\begin{eqnarray}
&& f_2(x,\epsilon_2) = - f_1(x,\epsilon_1) - 
\frac{\epsilon_1 - \epsilon_2}{f_1(x,\epsilon_1) - f_1(x,\epsilon_2)}.
\end{eqnarray}
Moreover, the solution $\psi_{1\,2}(x)$ to the Schr\"odinger equation (\ref{nSe}) is obtained by acting 
$A_1^+$ on the corresponding solution $\psi_{0\,2}(x)$ to the initial Schr\"odinger 
equation~(\ref{iSe}) associated to $\epsilon_2$, namely
\begin{eqnarray}
&& \psi_{1\,2}(x) = A_1^+ \psi_{0\,2} = - \frac{W(\psi_{0\,1},\psi_{0\,2})}{\psi_{0\,1}(x)},
\end{eqnarray}
where $W(\psi_{0\,1},\psi_{0\,2})$ is the Wronskian of the two seed solutions $\psi_{0\,1}(x),$ 
$\psi_{0\,2}(x)$.

By iterating $k$ times this procedure, a chain of intertwined Hamiltonians
\begin{equation}
H_{i+1} A_{i+1}^+ = A_{i+1}^+ H_i \quad \Leftrightarrow  \quad 
H_{i} A_{i+1}^- = A_{i+1}^- H_{i+1}, \quad i=0,1,\dots,k-1,
\end{equation}
is now constructed, where
\begin{eqnarray}
&& H_{i} = - \frac{d^2}{dx^2} + V_{i}(x), \\
&& A_{i+1}^\pm = \mp \frac{d}{dx} + f_{i+1}(x,\epsilon_{i+1}).
\end{eqnarray}
The superpotential $f_{i+1}(x,\epsilon_{i+1})$ must satisfy the Riccati equation
\begin{eqnarray}\label{kRe}
&& f_{i+1}'(x,\epsilon_{i+1}) + f_{i+1}^2(x,\epsilon_{i+1}) = V_i(x) - \epsilon_{i+1},
\end{eqnarray}
which is equivalent to the Schr\"odinger equation appearing by substituting 
$f_{i+1}(x,\epsilon_{i+1}) = \psi_{i\,i+1}'(x)/\psi_{i\,i+1}(x)$:
\begin{eqnarray}\label{kSe}
&& - \psi_{i\,i+1}''(x) + V_i(x) \psi_{i\,i+1}(x) = \epsilon_{i+1} \psi_{i\,i+1}(x).
\end{eqnarray}
Once again, $f_{i+1}(x,\epsilon_{i+1})$ is determined from two solutions
$f_{i}(x,\epsilon_{i})$, $f_{i}(x,\epsilon_{i+1})$ of the $i$th Riccati 
equation as follows:
\begin{eqnarray}
&& f_{i+1}(x,\epsilon_{i+1}) = - f_i(x,\epsilon_i) - 
\frac{\epsilon_i - \epsilon_{i+1}}{f_i(x,\epsilon_i) - f_i(x,\epsilon_{i+1})}.
\end{eqnarray}
Moreover, the Schr\"odinger solution $\psi_{i\,i+1}$ is obtained by acting $A_i^+$ on $\psi_{i-1\,i+1}$:
\begin{eqnarray}
&& \psi_{i\,i+1} = A_i^+ \psi_{i-1\,i+1} = - \frac{W(\psi_{i-1\,i},\psi_{i-1\,i+1})}{\psi_{i-1\,i}(x)}.
\end{eqnarray}

Let us note that, when iterating the last two formulas for decreasing indexes in order to generate a 
final potential $V_k(x)$ from the initial one $V_0(x)$, at the end we require to know $k$ solutions 
$f_1(x,\epsilon_{i}), \ i=1,\dots,k$ to the initial Riccati equation (\ref{iRe}). The same applies for 
the $k$ seed solutions $\psi_{0\,i}, i=1,\dots,k$, of the initial Schr\"odinger equation. In particular, 
the final potential $V_k(x)$ expressed in terms of these $k$ seed solutions is simply
\begin{eqnarray}
&& V_k(x) = V_0(x) - 2 [\ln W(\psi_{0\,1}, \dots, \psi_{0\,k})]'' ,
\end{eqnarray}
where $W(\psi_{0\,1}, \dots, \psi_{0\,k})$ denotes de Wronskian of $\psi_{0\,i}(x), \, i=1,\dots,k$.

\section{Discrete Darboux transformations}

In this section we introduce the Darboux transformation for a discrete version of the one-dimensional 
Schr\"odinger equation. At the end of this chapter we also present as an example the free particle case. 
We show how to build explicitly the new potential (a discrete analogue of the completely transparent 
potential) using the methods outlined in this section.

Let $ H_0$ and $ H_1$ be the following two discrete versions of the Schr\"odinger operator 
(Hamiltonian)
\begin{align}
\label{2} &  H_0= -\bigtriangleup^2 +V_0(n),\\
\label{3} &  H_1= -\bigtriangleup^2 +V_1(n),
\end{align}
where $\{V_0\}$ and $\{V_1\}$ are real-valued sequences. It is well known the correspondence between the 
one-dimensional discrete Schr\"odinger equation 
\begin{equation}\label{4}
\left( -\bigtriangleup^2 +V_0(n)\right )\psi(n) =\lambda\psi(n)
\end{equation}
with the matrix difference equation
\begin{equation}\label{5}
\bigtriangleup \left( \begin{array}{c}
\psi(n)\\
\bigtriangleup \psi(n)
\end{array}\right)= 
\left( \begin{array}{cc}
0 & 1\\
V_0(n)-\lambda & 0
\end{array}\right)   \left( \begin{array}{c}
\psi(n)\\
\bigtriangleup \psi(n)
\end{array}\right),
\end{equation}
and with the discrete Riccati equation
\begin{equation}\label{6}
\bigtriangleup f(n)+f(n)f(n+1)=V_0(n)-\lambda,
\end{equation}
whose solutions are related by
\begin{equation}\label{7}
f(n)=\dfrac{\bigtriangleup \psi(n)}{\psi(n)},
\end{equation}
where $\{\psi\}$ is a sequence and $\lambda$ is a constant. Moreover, let us suppose the existence of a 
first-order difference operator of the form
\begin{equation}
\label{8}
 \left( A^{(n)}_1\right)^+=-\bigtriangleup +f_1(n).
\end{equation}

We will implement next the first order Darboux transformation for the discrete Schr\"odinger equation 
through the following theorem.
\begin{theorem}
Let $ H_0$, $ H_1$ and $ \left( A^{(n)}_1\right)^+$ be the operators defined by (\ref{2}), (\ref{3}) and 
(\ref{8}). If 
\begin{equation}
\label{9}
 H_1  \left( A^{(n)}_1\right)^+= \left( A^{(n+2)}_1\right)^+  H_0,
\end{equation}
then the sequences $\{V_0\}$, $\{V_1\}$, $\{f_1\}$ satisfy
\begin{align}
\label{10} & V_1(n)=V_0(n+1)-2 \bigtriangleup f_1(n+1),\\
\label{11} & -\bigtriangleup^2 f_1(n)+\bigtriangleup V_0(n)-2f_1(n)\bigtriangleup f_1(n+1)=\\
& -f_1(n)V_0(n+1)+f_1(n+2)V_0(n).\nonumber
\end{align} 
We will say that the operator $\left( A^{(n)}_1\right)^+$ intertwines the two Hamiltonians $H_0$ and 
$H_1$.
\end{theorem}

\begin{proof}
We can write equality (\ref{9}) in the form
\begin{equation}
\left(-\bigtriangleup^2 +V_1(n) \right)\left( -\bigtriangleup +f_1(n)\right)=
\left(-\bigtriangleup +f_1(n+2)\right)\left(-\bigtriangleup^2 +V_0(n) \right).\nonumber
\end{equation}
Simple calculations using equation (\ref{1}) yield
\begin{equation}
\left( -(T^{+})^2+2  T^{+}-1+V_1(n)\right) \left(- T^{+}+1+f_1(n)\right)=\nonumber
\end{equation}
\begin{equation}
=  \left(- T^{+}+1+f_1(n+2)\right) \left( - (T^{+})^2+2  T^{+}-1+V_0(n)\right), \nonumber
\end{equation}
\begin{equation}
\left( -V_1(n)+V_0(n+1)+2f_1(n+1)-2f_1(n+2) \right) T^{+} +\nonumber
\end{equation}
\begin{equation}
f_1(n+2)-f_1(n)+V_1(n)\left( 1+f_1(n)\right)-V_0(n)\left(1+f_1(n+2)\right)= 0. \nonumber
\end{equation}
By collecting the coefficients of $ T^{+}$ and the identity operator, which must vanish independently of 
each other, we obtain equations (\ref{10}) and (\ref{11}) respectively. This finishes the proof.
\end{proof}

Let us stress the importance that the apparently odd equation (\ref{9}) has in our treatment for generating 
non-trivial new potentials $V_1(n)$. In fact, if instead of Eq.~(\ref{9}) we would ask that  
$ H_1  \left( A^{(n)}_1\right)^+= \left( A^{(n)}_1\right)^+  H_0$, 
then we would obtain just the trivial result $V_1(n) = V_0(n+1) = {\rm constant}$.

\begin{theorem}
If 
\begin{equation}\label{12a}
f_1(n)=\frac{\bigtriangleup \psi_1(n)}{\psi_1(n)},
\end{equation} where
\begin{equation}\label{12}
\quad \left( -\bigtriangleup^2 +V_0(n)\right )\psi_1(n) =\epsilon \psi_1(n+2),\quad \epsilon \in\mathbb{R},
\end{equation}
then the condition (\ref{11}) is fulfilled.
\end{theorem}

\begin{proof}
Through the substitution $f_1(n)=\frac{\bigtriangleup \psi_1(n)}{\psi_1(n)}$, the requirement (\ref{11}) 
is equivalent to the equation
\begin{equation}
\dfrac{\psi_1(n+3)}{\psi_1(n+2)}-\dfrac{\psi_1(n+1)}{\psi_1(n)}+2\dfrac{\psi_1(n+2)}{\psi_1(n)}-2\dfrac{\psi_1(n+1)\psi_1(n+3)}{\psi_1(n)\psi_1(n+2)}+ \nonumber
\end{equation}
\begin{equation}
V_0(n+1)\dfrac{\psi_1(n+1)}{\psi_1(n)}-V_0(n)\dfrac{\psi_1(n+3)}{\psi_1(n+2)}=0.\nonumber
\end{equation}
It can  be rewritten now in the form
\begin{equation}
\psi_1(n+3)\bigg( \psi_1(n+2)-2 \psi_1(n+1)+ \psi_1(n)-V_0(n) \psi_1(n)\bigg)- \nonumber
\end{equation}
\begin{equation}
\psi_1(n+2)\bigg( \psi_1(n+3)-2 \psi_1(n+2)+ \psi_1(n+1)-V_0(n+1) \psi_1(n+1)\bigg)=0.\nonumber
\end{equation}
Finally, we obtain equation (\ref{12}) for the sequence $\{\psi_1\}$,
\begin{equation}
\psi_1(n+2)-2 \psi_1(n+1)+ \psi_1(n)-V_0(n) \psi_1(n)=-\epsilon \psi_1(n+2).\nonumber
\end{equation}
\end{proof}

\begin{corollary}
\label{Proposition 3}
Let $\{\varphi_0\}$ be a solution of 
\begin{equation}\label{13}
H_0\varphi_0(n)=\bigg(-\bigtriangleup^2 +V_0(n)\bigg)\varphi_0(n)=0.
\end{equation} Then $\varphi_1(n)= \left( A^{(n)}_1\right)^+\varphi_0(n)$ is a solution of 
\begin{equation}\label{14}
H_1\varphi_1(n)=\bigg(-\bigtriangleup^2 +V_1(n)\bigg)\varphi_1(n) =0.
\end{equation}
\end{corollary}

\begin{proof}
This is a direct consequence of formula (\ref{9}).
\end{proof}

Let us note that if we take $\epsilon =0$ and $f_1(n) = \bigtriangleup \varphi_0(n)/\varphi_0(n)$, where 
$\varphi_0(n)$ satisfies equation (\ref{13}), then equations (\ref{10}), (\ref{12}) guarantee that $H_0$ 
and $H_1$ become factorized in the form
\begin{align}
  &\label{H_0}  H_0=\bigg(\bigtriangleup +f_1(n+1)\bigg)\bigg(-\bigtriangleup +f_1(n)\bigg),\\
  &\label{H_1} H_1=\bigg(-\bigtriangleup +f_1(n+2)\bigg)\bigg(\bigtriangleup +f_1(n+1)\bigg),
\end{align}
i.e. the potentials can be written in the form
\begin{align}
&\label{V_0} V_0(n)= \bigtriangleup f_1(n)+f_1(n)f_1(n+1), \\
&\label{V_1} V_1(n)= -\bigtriangleup f_1(n+1)+f_1(n+1)f_1(n+2).
\end{align}
In addition, from equations (\ref{12a}) and (\ref{12}) we can see that for solutions $\psi_1(n)$ 
satisfying equation (\ref{12}) with $\epsilon \neq 0$, the potential $V_0$ can be also expressed as
\begin{equation} \label{V0ne0}
V_0(n)= (1+ \epsilon )\big( \bigtriangleup f_1(n)+f_1(n)f_1(n+1) \big)+\epsilon  \big( 2 f_1(n)+1 \big).
\end{equation} 
Note that the $f_1(n)$ of equations (\ref{V_0}) and (\ref{V0ne0}) are not the same. Using now the expression
\begin{equation}
f_1(n)=\dfrac{w(n+1)}{(1+\epsilon) w(n)}-1, \quad \epsilon \neq -1
\end{equation}
equation (\ref{V0ne0}) transforms into
\begin{equation}\label{12-n}
\quad \bigg( -\bigtriangleup^2 +(1+\epsilon)V_0(n)\bigg)w(n) =\epsilon w(n),
\end{equation}
which means that the change of variables $\psi_1(n)=\frac{C}{(1+\epsilon)^n}w(n)$ transforms equation 
(\ref{12}) into (\ref{12-n}), where $C$ is a constant and $\epsilon\neq -1$. The case when $\epsilon=-1$ 
is not very interesting because then $V_1(n)=V_0(n+2)$.

\begin{example}
{\bf Completely transparent potential for the discrete Schr\"o\-dinger equation.}
Let us consider the case when $V_0(n)=0$. In this example the general solution of the homogeneous linear 
second-order difference equations with constant coefficients (\ref{13}),
\begin{equation}
-\bigtriangleup^2\varphi_0(n)=0,
\end{equation}
 is given by
\begin{equation}
\label{17}
\varphi_0(n)=C_1+C_2n,
\end{equation}
where $C_1$ and $C_2$ are constants. Next, we will look for some particular solutions of equation 
(\ref{12}), i.e.
\begin{equation}
\label{18}
-\bigtriangleup^2 \psi_1(n)=\epsilon \psi_1(n+2).
\end{equation}
After substituting $\psi_1(n)=\lambda^n$ we obtain the characteristic equation of this difference equation
\begin{equation}
\label{19}
-(1+\epsilon )\lambda^2+2\lambda-1=0.
\end{equation}
The set of solutions of this equation depends on the character of the roots of the characteristic equation 
$\Delta= -4\epsilon$ as follows. 
\begin{enumerate}[a)]
\item Single real root. If $\epsilon =0$, then $\Delta=0$ and we can choose 
\begin{equation}
\label{20}
\psi_1(n)=n+1.
\end{equation}
From equations (\ref{12a}) and (\ref{10}) we thus find
\begin{align}
\label{21}
& f_1(n)=\dfrac{1}{n+1},\\
& V_1(n)=\dfrac{2}{(n+2)(n+3)}.
\end{align}
Moreover, from Corollary  \ref{Proposition 3} we obtain that
\begin{equation}
\label{22}
\varphi_1(n)=\dfrac{C_1-C_2}{n+1}
\end{equation}
is a solution of the following equation 
\begin{equation}
\label{23}
H_1 \varphi_1(n)=\bigg(-\Delta^2 + \dfrac{2}{(n+2)(n+3)}\bigg) \varphi_1(n)=0.
\end{equation}

\item Distinct real roots. If $\epsilon <0$ (we assume $\epsilon =-\kappa^2$ and $\kappa \neq -1,1$), 
then $\Delta>0$ and we can choose 
\begin{equation}
\label{24}
\psi_1(n)=\dfrac{1}{(1+\kappa)^{n}}+\dfrac{1}{(1-\kappa)^{n}}.
\end{equation}
From equations (\ref{12a}) and (\ref{10}) we thus find
\begin{align}
\label{25}
& f_1(n)=\dfrac{\kappa}{1-\kappa^2}\dfrac{(1+\kappa)^{n+1}-(1-\kappa)^{n+1}}{(1+\kappa)^{n}+(1-\kappa)^{n}},\\
& V_1(n)=\dfrac{-8\kappa^2 (1-\kappa^2)^{n}}
{\big((1+\kappa)^{n+1}+(1-\kappa)^{n+1}\big)\big( (1+\kappa)^{n+2}+(1-\kappa)^{n+2} \big)}.
\end{align}
Moreover, from Corollary \ref{Proposition 3} we obtain that
\begin{equation}
\label{26}
\varphi_1(x)=- C_2+
\dfrac{\kappa (C_1+C_2n)\big(  (1+\kappa)^{n+1}-(1-\kappa)^{n+1}\big)}
{\big( 1-\kappa^2\big)\big( (1+\kappa)^{n}+(1-\kappa)^{n} \big)}
\end{equation}
is a solution of the following equation 
\begin{equation}
\label{27}
\bigg(-\Delta^2 - \dfrac{8\kappa^2 (1-\kappa^2)^{n}}
{\big((1+\kappa)^{n+1}+(1-\kappa)^{n+1}\big)\big( (1+\kappa)^{n+2}+(1-\kappa)^{n+2} \big)}\bigg) \varphi_1(n)=0.
\end{equation}

Note that when $\kappa = -1$ or $\kappa = 1$ we obtain that equation (\ref{18}) reduces to
\begin{equation}
\label{28}
2 \psi_1(n+1)= \psi_1(n).
\end{equation}
This leads to $\psi_1(n)=\frac{B}{2^n}$, where $B$ is a constant, which implies that the new potential 
does not change, $V_1(n)=0$.

\item Complex roots. If $\epsilon >0$ (we assume $\epsilon =\mu^2$), then $\Delta<0$ and we can choose 
\begin{equation}
\label{29}
\psi_1(n)=r^n\cos n\theta  ,
\end{equation}
where $r=\frac{1}{\sqrt{1+\mu^2}}$ and $\tan\theta=\mu$. From equations (\ref{12a}) and (\ref{10}) we find
\begin{align}
\label{30}
& f_1(n)=\dfrac{1}{\sqrt{1+\mu^2}} \dfrac{\cos (n+1)\theta }{\cos n\theta }-1,\\
& V_1(n)=-\dfrac{2}{\sqrt{1+\mu^2}} \dfrac{\cos (n+1)\theta \cos (n+3)\theta -\cos^2 (n+2)\theta }{\cos (n+1)\theta  \cos (n+2)\theta}.
\end{align}
Moreover, from Corollary \ref{Proposition 3} we obtain that
\begin{equation}
\label{31}
\varphi_1(x)=-C_2 n-(C_1+C_2)+\dfrac{1}{\sqrt{1+\mu^2}} \dfrac{(C_1+C_2 n)\cos (n+1)\theta }{\cos n\theta }
\end{equation}
is a solution of the equation 
\begin{equation}
\label{32}
\bigg(-\Delta^2 - \dfrac{2}{\sqrt{1+\mu^2}} \dfrac{\cos (n+1)\theta \cos (n+3)\theta -\cos^2 (n+2)\theta }{\cos (n+1)\theta  \cos (n+2)\theta} \bigg) \varphi_1(n)=0.
\end{equation}

\end{enumerate}
\end{example}

\begin{example} {\bf Discrete version of the harmonic oscillator.}
In the case when it is taken $f_1(n)=n$ in equation (\ref{V_0}) we obtain that $V_0(n)=n^2+n+1$. Thus, 
our next natural step is to look for the solution of equation (\ref{13}) with such a $V_0(n)$, i.e.
\begin{equation}
\label{osc}
\bigg(-\bigtriangleup^2+n^2+n+1\bigg)\varphi_0(n)=0.
\end{equation} 
The right-hand side of equation (\ref{H_0}) suggests us to use the following simpler equation
\begin{equation}
\bigg(-\bigtriangleup+n\bigg)\varphi_0(n)=0.
\end{equation}
A straightforward computation yields $\varphi_0(n)=C n!$, where $C$ is a constant. This is also a 
particular solution of equation (\ref{osc}). As in the continuous case \cite{Mielnik-2,fe84}, now we 
look for the general solution of the discrete Riccati equation (\ref{V_0}) using the transformation
\begin{equation}
\tilde{f_1}(n)=f_1(n)+\dfrac{1}{u(n)},
\end{equation}
where $u(n)$ must satisfy the first order difference equation
\begin{equation}
\bigtriangleup u(n)=f_1(n)u(n)+f_1(n+1)u(n+1)+1,
\end{equation}
\begin{equation}
\label{u(n)}
-n u(n+1)=(n+1)u(n)+1.
\end{equation}
A straightforward calculation shows that the solution of equation (\ref{u(n)}) becomes
\begin{equation}
u(n)=(-1)^n n\bigg( C_1+\sum_{i=1}^{n-1}\dfrac{(-1)^i}{i(i+1)} \bigg),
\end{equation}
where $C_1$ is a constant. The above expression leads to the general solution of the Riccati equation 
(\ref{V_0}) we were looking for:
\begin{equation}
\tilde{f_1}(n)=n+\dfrac{(-1)^n}{C_1 n+(-1)^n +2n\sum_{i=3}^{n-1}\frac{(-1)^i}{i} }.
\end{equation}
Finally, the new potential is produced by using equation (\ref{10}):
\begin{equation}
V_1(n)=n^2+3n+1
\end{equation}
$$
+2\dfrac{(-1)^n C_1(2n+3)+(-1)^n (2n+3) \sum_{i=1}^{n}\dfrac{(-1)^i}{i(i+1)}+\frac{1}{n+1}}
{(n+1)(n+2)\bigg( C_1+\sum_{i=1}^{n}\dfrac{(-1)^i}{i(i+1)} \bigg)\bigg( C_1+\sum_{i=1}^{n+1}\dfrac{(-1)^i}{i(i+1)} \bigg)}.
$$

\end{example}

\section{Discrete Crum's theorem}

In this section we introduce an analogue of the Crum's transformation for the discrete one-dimensional 
Schr\"odinger equation. At the end of this section we also present an example related to the free 
particle case. 
  
Let us note first of all that if we take a solution $\psi_i$ of equation (\ref{12}) for  a constant 
$\epsilon _i$ and apply to it the operator sequence (\ref{9}), we conclude that
\begin{equation}
\label{41}
 H_1 \bigg(-\bigtriangleup +f_1(n)\bigg)\psi_i(n)=\bigg(-\bigtriangleup +f_1(n+2)\bigg)  H_0 \psi_i(n)
\end{equation}
$$
=\epsilon_i \bigg(-\bigtriangleup +f_1(n+2)\bigg) \psi_i(n+2).
$$
Thus, the sequence 
\begin{equation}
\label{42}
\tilde{\psi}_i(n)= \big(-\bigtriangleup +f_1(n)\big)\psi_i(n)=
\dfrac{\psi_i(n)\bigtriangleup \psi_1(n)-\psi_1(n)\bigtriangleup \psi_i(n)}{\psi_1(n)}
\end{equation}
$$
=\dfrac{\psi_1(n+1) \psi_i(n) -\psi_1(n) \psi_i(n+1)}{\psi_1(n)}
$$$$
=-\dfrac{\left|
\begin{array}{cc}
\psi_1(n) & \psi_i(n)\\
\bigtriangleup \psi_1(n) & \bigtriangleup \psi_i(n)
\end{array}\right| }{\psi_1(n)}
=-\dfrac{\left|
\begin{array}{cc}
\psi_1(n) & \psi_i(n)\\
 \psi_1(n+1) & \psi_i(n+1)
\end{array}\right| }{\psi_1(n)}
$$
satisfies an equation with the same form as equation (\ref{12}) but for the new potential $V_1$,
\begin{equation}
\left( -\bigtriangleup^2 +V_1(n)\right )\tilde{\psi}_i(n) =\epsilon_i \tilde{\psi}_i(n+2).
\end{equation}  

Let us introduce now a well-know notation in theory of difference equations. The Casortian of the 
solutions  $\psi_1(n), \psi_2(n),\dots, \psi_k(n)$ is defined by
\begin{equation}
\!\!C(\psi_1, \dots, \psi_k)(n)\!\!:=\!\!\left|\!\!
\begin{array}{cccc}
\psi_1(n) & \psi_2(n) & \dots & \psi_k(n)\\
\psi_1(n+1) & \psi_2(n+1) & \dots & \psi_k(n+1)\\
\vdots & \vdots & \ldots & \dots\\
\psi_1(n+k-1) \!\!& \psi_2(n+k-1) \!\!& \dots \!\!& \psi_k(n+k-1)\\
\end{array}\!\!\right|\!,
\end{equation} 
see e.g. \cite{GaiMat}.
So, we can write formula (\ref{42}) in the form 
\begin{equation}
\tilde{\psi}_i(n)=-\dfrac{C(\psi_1, \psi_i)(n)}{C(\psi_1)(n)}.
\end{equation}    
  
Now, let us apply iteratively the technique from the previous section. We consider the new intertwining 
relation
\begin{equation}
 H_2 \bigg(-\bigtriangleup +f_2(n)\bigg)=\bigg(-\bigtriangleup +f_2(n+2)\bigg)  H_1 ,
\end{equation}
which leads to equations similar to (\ref{10}), (\ref{11}) (or (\ref{12}))
\begin{align}
\label{45} & V_2(n)=V_1(n+1)-2 \bigtriangleup f_2(n+1),\\
\label{46} & -\bigtriangleup^2 f_2(n)+\bigtriangleup V_1(n)-2f_2(n)\bigtriangleup f_2(n+1)=\\
& -f_2(n)V_1(n+1)+f_2(n+2)V_1(n).\nonumber
\end{align} 
If we choose $\psi_2$ as a solution of equation (\ref{12}), then  from expression (\ref{42}) we obtain
\begin{equation}
f_2(n)=\dfrac{\bigtriangleup \tilde{\psi}_2(n)}{\tilde{\psi}_2(n)}
=\dfrac{\bigtriangleup \big(-\bigtriangleup +f_1(n)\big)\psi_2(n)}{\big(-\bigtriangleup +f_1(n)\big)\psi_2(n)}
\end{equation}
$$
=\dfrac{\psi_1(n)}{\psi_1(n+1)}\dfrac{\psi_1(n+2)\psi_2(n+1)-\psi_1(n+1)\psi_2(n+2)}{\psi_1(n+1)\psi_2(n)-\psi_1(n)\psi_2(n+1)}-1
$$$$
=\dfrac{C(\psi_1)(n)}{C(\psi_1)(n+1)}\dfrac{C(\psi_1, \psi_2)(n+1)}{C(\psi_1, \psi_2)(n)}-1.
$$
From this result and equation (\ref{10}) the new potential is found,
\begin{equation}\label{87}
V_2(n)=V_1(n+1)-2 \bigtriangleup f_2(n+1)=V_0(n+2)-2 \bigtriangleup\big( f_1(n+2) + f_2(n+1)\big)
\end{equation}
$$
=V_0(n+2)-2 \bigtriangleup\bigg( \dfrac{\psi_1(n+3)\psi_2(n+1)-\psi_1(n+1)\psi_2(n+3)}{\psi_1(n+2)\psi_2(n+1)-\psi_1(n+1)\psi_2(n+2)}-2\bigg)
$$ $$
=V_0(n+2)-2 \bigtriangleup
\dfrac{\left|
\begin{array}{cc}
\psi_1(n+1) & \psi_2(n+1)\\
 \psi_1(n+3) &  \psi_2(n+3)
\end{array}\right| }
{C(\psi_1,\psi_2)(n+1) } .
$$
It is easy to see that the Bianchi property is fulfilled, i.e., if we interchange the seed solutions 
$\psi_1$ by $\psi_2$ and vice versa, we will obtain the same final potential $V_2$.

This iterative process can be continued at will. The third step
\begin{equation}
H_3 \bigg(-\bigtriangleup +f_3(n)\bigg)=\bigg(-\bigtriangleup +f_3(n+2)\bigg)  H_2 ,
\end{equation}
produces similar results 
 \begin{align}
\label{50} & V_3(n)=V_2(n+1)-2 \bigtriangleup f_3(n+1),\\
\label{51} & -\bigtriangleup^2 f_3(n)+\bigtriangleup V_2(n)-2f_3(n)\bigtriangleup f_3(n+1)=\\
& -f_3(n)V_2(n+1)+f_3(n+2)V_2(n).\nonumber
\end{align}
Similarly to expression (\ref{41}), it is easy to see that if $\psi_i$ satisfies equation (\ref{12}) for 
$\epsilon_i$, then the sequence
\begin{equation}
\label{x_3}
\tilde{\tilde{\psi}}_i(n)=\big(-\bigtriangleup +f_2(n)\big)\big(-\bigtriangleup +f_1(n)\big)\psi_i(n)
\end{equation}
$$
=\dfrac{\left|
\begin{array}{ccc}
\psi_1(n) & \psi_2(n) & \psi_i(n)\\
 \psi_1(n+1) &  \psi_2(n+1) &  \psi_i(n+1)\\
 \psi_1(n+2) &  \psi_2(n+2) &  \psi_i(n+2)\\
\end{array}\right| }
{\left|
\begin{array}{cc}
\psi_1(n) & \psi_2(n)\\
\psi_1(n+1) &  \psi_2(n+1)
\end{array}\right| }=\dfrac{C(\psi_1,\psi_2, \psi_i)(n)}{C(\psi_1,\psi_2)(n)}
$$
satisfies the new equation
\begin{equation}
\quad \left( -\bigtriangleup^2 +V_2(n)\right )\tilde{\tilde{\psi}}_i(n) =\epsilon_i \tilde{\tilde{\psi}}_i(n+2).
\end{equation}
Moreover, if we choose $\psi_3$ as a solution of equation (\ref{12}) then from expression (\ref{x_3}) we 
obtain
\begin{equation}
f_3(n)=\dfrac{\bigtriangleup \tilde{\tilde{\psi}}_3(n)}{\tilde{\tilde{\psi}}_3(n)}
=\dfrac{\bigtriangleup \left[ \big(-\bigtriangleup +f_2(n)\big)\big(-\bigtriangleup +f_1(n)\big)\psi_3(n)\right]}{\big(-\bigtriangleup +f_2(n)\big) \big(-\bigtriangleup +f_1(n)\big)\psi_3(n)}
\end{equation}
$$
=\dfrac{C(\psi_1, \psi_2)(n)}{C(\psi_1,\psi_2)(n+1)}\dfrac{C(\psi_1, \psi_2, \psi_3)(n+1)}{C(\psi_1, \psi_2, \psi_3)(n)}-1,
$$
and 
\begin{equation}
V_3(n)=V_0(n+3)-2\bigtriangleup \left( f_1(n+3)+f_2(n+2)+f_3(n+1)\right)
\end{equation}
$$
=V_0(n+3)-2\bigtriangleup 
\dfrac{\left|
\begin{array}{ccc}
\psi_1(n+1) & \psi_2(n+1) & \psi_3(n+1)\\
 \psi_1(n+2) &  \psi_2(n+2) &  \psi_3(n+2)\\
 \psi_1(n+4) &  \psi_2(n+4) &  \psi_3(n+4)\\
\end{array}\right| }
{C(\psi_1, \psi_2, \psi_3)(n+1)}.
$$

The previous results allow to formulate next a theorem for the $k$th iteration in a compact 
form, since we have
\begin{equation}
H_i \big( A_i^{(n)}\big)^{+}=
\big( A_i^{(n+2)}\big)^{+}  H_{i-1} ,\quad i=1,2,\dots, k,
\end{equation}
which leads to the following higher-order intertwining relationships
\begin{equation}
\label{rel}
H_i \big( A_i^{(n)}\big)^{+}\big( A_{i-1}^{(n)}\big)^{+}\dots \big( A_1^{(n)}\big)^{+}=
\big( A_i^{(n+2)}\big)^{+}\big( A_{i-1}^{(n+2)}\big)^{+}\dots \big( A_1^{(n+2)}\big)^{+}  H_{0} .
\end{equation}

\begin{theorem}
If the seed solutions $\psi_1(n), \psi_2(n),\dots, \psi_k(n)$ satisfy the initial equation (\ref{12}) for 
different constants $\epsilon_1, \epsilon_2, \dots, \epsilon_k$, then the functions
\begin{equation}
\hat{ \psi}_i(n)=\prod _{j=1}^{i-1}\left(-\bigtriangleup +f_j(n) \right)\psi_i(n)
=(-1)^{i-1}\dfrac{C(\psi_1, \dots, \psi_i)(n)}{C(\psi_1, \dots, \psi_{i-1})(n)}, 
\end{equation} 
$i=1,2,\dots,k,$, satisfy the equations
\begin{equation}
\quad \left( -\bigtriangleup^2 +V_{i-1}(n)\right )\hat{ x}_i(n) =\epsilon_i \hat{ x}_i(n+2),
\end{equation}
where
\begin{equation}
f_j(n)=\dfrac{C(\psi_1, \dots, \psi_{j-1})(n)}{C(\psi_1, \dots, \psi_{j-1})(n+1)}
\dfrac{  C(\psi_1, \dots, \psi_{j})(n)}{C(\psi_1, \dots, \psi_{j})(n+1)}
-1,
\end{equation}
\begin{equation}
V_i(n)=V_0(n+i)-2\bigtriangleup \left( f_1(n+i)+f_2(n+i-1)+\dots +f_i(n+1)\right)
\end{equation}
$$
=V_0(n+i)-2\bigtriangleup 
\dfrac{\left|
\begin{array}{cccc}
\psi_1(n+1) & \psi_2(n+1) & \dots & \psi_i(n+1)\\
 \psi_1(n+2) &  \psi_2(n+2) &  \dots & \psi_i(n+2)\\
 \dots& \dots & \dots & \dots\\
  \psi_1(n+i) &  \psi_2(n+i) &  \dots & \psi_i(n+i)
\end{array}\right| }
{C(\psi_1, \psi_2,\dots, \psi_i)(n+1)}.
$$
\end{theorem}

\begin{proof}
We have given the proof previously for $n =1,2,3$. For other values of $n$, it is based on a proof by 
induction and the  observation contained in equation (\ref{rel}).
\end{proof}

\begin{example}
{\bf Free particle: discrete second order Darboux transformation.} Let us apply the previous treatment 
to the free particle, for which $V_0(n)=0$. In order to avoid unnecessarily long formulas in our final 
results, we fix the two seed solutions of the initial Schr\"odinger problem as follows
\begin{eqnarray}
&& \psi_1(n) = \frac{1}{n+1}, \\
&& \psi_2(n) = \frac{2^n(1+3^n)}{3^n},
\end{eqnarray} 
which are associated to $\epsilon_1=0, \ \epsilon_2 = -1/4$ respectively (see Example 1 with 
$\kappa=1/2$). We use equation (\ref{87}) for calculating the new potential; we obtain:
\begin{eqnarray}
&& V_2(n) = 2\frac{3^{2n+5} + 4(2n^2 + 14n + 27)3^{n+1} + 1}{[n+5-(n+1)3^{n+2}] [n+6 -(n+2)3^{n+3}]}.
\end{eqnarray} 
By selecting then a third initial solution, for example
\begin{equation}
\psi_3(n)=C_1+C_2n,
\end{equation}
we find that the transformed function 
\begin{equation}
\hat{ \psi}_3(n)=\left(-\bigtriangleup +2\dfrac{n+1}{n+2}\dfrac{(3^{-1}-3^{n+1})n+\frac{5}{3}-3^{n+1}}{(1+3^{n+1})n+4} \right)
\end{equation}
$$
\left(-\bigtriangleup +\dfrac{1}{n+1} \right)\left(C_1+C_2n\right)
$$
 is a solution of equation
\begin{equation}
\left(-\bigtriangleup^2+V_2(n)\right) \hat{ \psi}_3(n)=0.
\end{equation}
\end{example}

\bibliographystyle{plain}

\end{document}